\documentclass[a4paper]{amsart}

\usepackage[utf8]{inputenc}
\usepackage[english]{babel}

\usepackage[T1]{fontenc}
\usepackage{lmodern}  

\usepackage{my-math-symbol}
\usepackage{my-cleveref}
\usepackage{my-equation-numbering}
\usepackage{comment}
\usepackage{tikz}
\usetikzlibrary{cd}
\usepackage[colorinlistoftodos]{todonotes}

\usepackage[non-sorted-cites,initials,alphabetic,nobysame]{amsrefs}

\AtBeginDocument{%
	\def\MR#1{}
}

\title{CR Yamabe constant and inequivalent CR structures}
\author{Chanyoung Sung}
\address{Department of Mathematics education, Korea National University of Education, Cheongju, Korea}
\email{cysung@kias.re.kr}

\author{Yuya Takeuchi}
\address{Division of Mathematics, Institute of Pure and Applied Sciences, University of Tsukuba, 1-1-1 Tennodai, Tsukuba, Ibaraki 305-8571 Japan}
\email{ytakeuchi@math.tsukuba.ac.jp, yuya.takeuchi.math@gmail.com}
\thanks{The second author was supported by JSPS KAKENHI Grant Number JP21K13792.}

\subjclass[2020]{32V20, 32V05, 57K43}

\keywords{CR Yamabe constant, CR manifold, pseudohermitian manifold}

\begin{document}

\begin{abstract}
	The CR Yamabe constant is an invariant of a compact strongly pseudoconvex CR manifold
	and plays an important role in CR geometry.
	We show some integral formulae of the CR Yamabe constant.
	We also construct an infinite-dimensional family of strongly pseudoconvex CR structures
	with varying CR Yamabe constants
	and a compact simply-connected manifold
	admitting two strongly pseudoconvex CR structures
	with different signs of the CR Yamabe constant.
\end{abstract}

\maketitle

\section{Introduction}
\label{section:introduction}

The Yamabe problem,
which is one of the most important problems in conformal geometry,
asks whether there exists a Riemannian metric in a given conformal class
minimizing the Yamabe functional.
The infimum of this functional defines a conformal invariant called the Yamabe constant.
The Yamabe constant is a fundamental invariant in conformal geometry,
and there are intensive research on this invariant.
It is known that every compact manifold of dimension greater than $2$
has a continuous family of conformal structures with all different Yamabe constants.
Moreover,
every compact manifold of dimension greater than $2$
admits a conformal structure with negative Yamabe constant.
Furthermore,
there exist some integral formulae of this invariant,
which have been useful for computing not only the Yamabe constants
but also other curvature parts of $3$- and $4$-manifolds~\cites{LeBrun1999,Sung2012,Sung2021}.

Jerison and Lee~\cite{Jerison-Lee1987} have considered a CR analog of the Yamabe problem,
known as the \emph{CR Yamabe problem}.
The CR Yamabe problem asks whether
on any compact strongly pseudoconvex CR manifold $(X, H, J)$ of dimension $2 n + 1$,
there exists a contact form $\theta$ minimizing the functional
\begin{equation}
	\mathfrak{F}(\theta)
	\coloneqq \frac{\int_{X} R_{\theta} \, d \mu_{\theta}}
		{(\int_{X} \, d \mu_{\theta})^{n / (n + 1)}},
\end{equation}
where $R_{\theta}$ is the pseudohermitian scalar curvature for $\theta$
and $d \mu_{\theta} = \theta \wedge (d \theta)^{n}$.
The infimum $\lambda(X, H, J)$ of the above functional is an invariant of the CR manifold $(X, H, J)$
and called the \emph{CR Yamabe constant} of $(X, H, J)$;
we will simply write $\lambda(X)$ when the CR structure on $X$ is clear from the context.
Just as in the Yamabe problem,
one has
\begin{equation}
	\lambda(X)
	\leq \lambda(S^{2 n + 1})
	= 2 n (n + 1) \pi
\end{equation}
for the standard CR structure on $S^{2 n + 1}$,
and the CR Yamabe problem is solvable when $\lambda(X) < \lambda(S^{2 n + 1})$~\cite{Jerison-Lee1987}.
There are intensive researches on conditions when $\lambda(X) < \lambda(S^{2 n + 1})$ holds
for $X$ not CR equivalent to $S^{2 n + 1}$;
see~\cites{Jerison-Lee1988, Jerison-Lee1989, Cheng-Chiu-Yang2014, Cheng-Malchiodi-Yang2017, Takeuchi2020-Paneitz, Cheng-Malchiodi-Yang2023} for related results.
Every minimizer $\theta$ has constant $R_{\theta}$,
and when $\lambda(X) \leq 0$,
any $\theta$ with constant $R_\theta$ is a CR Yamabe minimizer,
which is unique up to a constant.
Moreover,
an Einstein contact form is also a CR Yamabe minimizer;
see \cref{lem:Einstein-is-CR-Yamabe-minimizer}.
In a similar way to the Yamabe constant,
the CR Yamabe constant can be written as various integral formulae;
see \cref{thm:integral-formulae-of-CR-Yamabe-constant}.
This may be useful for estimating norms of parts of pseudohermitian curvature tensor
as in the Riemannian case~\cite{Sung2021}.

It is natural to ask whether a manifold admitting one CR structure
has abundant other CR structures with all different CR Yamabe constants
or CR structures with negative CR Yamabe constant.
In comparison to the Riemannian case,
the difficulty lies in imposing the integrability condition to an almost CR structure,
which obstructs generic deformations of almost CR structures.
In this paper,
we will construct an infinite-dimensional family of strongly pseudoconvex CR structures
with varying CR Yamabe constants.
To this end,
we consider an $n$-dimensional compact Hodge manifold $(M, J, \omega)$ with constant scalar curvature.
Denote the space of \Kahler potentials in the class $[\omega]$ by
\begin{equation}
	\calK
	\coloneqq \Set{\varphi \in C^{\infty}(M) | \omega_{\varphi} \coloneqq \omega + i \del \delb \varphi > 0}
\end{equation}
endowed with the $C^{4}$-topology,
and write $\calF$ for the subset of $\varphi \in \calK$
such that $\omega_{\varphi}$ has constant scalar curvature.
Let $p \colon P_{M} \to M$ be a principal $S^{1}$-bundle over $M$ whose Euler class is $- [\omega]$.
For any $\varphi \in \calK$,
there exists a principal connection $\theta_{\varphi}$ on $P_{M}$ such that
$d \theta_{\varphi} / 2 \pi = p^{\ast} \omega_{\varphi}$.
The complex structure $J$ on $M$ induces a strongly pseudoconvex CR structure $p^{\ast} J$
on $H_{\varphi} \coloneqq \Ker \theta_{\varphi}$;
see \cref{prop:circle-bundle-over-Hodge-manifold}.
This gives an infinite-dimensional family of pseudohermitian manifolds
$(P_{M}, H_{\varphi}, p^{\ast} J, \theta_{\varphi} / 2 \pi)$ underlying the same manifold $P_{M}$.

\begin{theorem}
\label{thm:varying-CR-Yamabe-constant}
	Let $(M, J, \omega)$ be an $n$-dimensional compact Hodge manifold.
	Then the map
	\begin{equation}
		\calK \to \bbR;
		\qquad
		\varphi \mapsto \lambda(P_{M}, H_{\varphi}, p^{\ast} J)
	\end{equation}
	is continuous.
	Moreover if $\theta_{0} / 2 \pi$ is a CR Yamabe minimizer,
	then
	\begin{equation}
		\lambda(P_{M}, H_{\varphi}, p^{\ast} J)
		< \lambda(P_{M}, H_{0}, p^{\ast} J)
		= \frac{2 n \pi c_{1}(M) \cup [\omega]^{n - 1}}{([\omega]^{n})^{n / (n + 1)}}.
	\end{equation}
	for any $\varphi \in \calK \setminus \calF$.
	The assumption holds if $\omega$ has constant non-positive scalar curvature
	or it defines a \Kahler-Einstein metric.
\end{theorem}

We will also show the existence of a compact simply-connected manifold
admitting two strongly pseudoconvex CR structures
with different signs of the CR Yamabe constant.

\begin{theorem}
\label{thm:inequivalent-CR-manifold}
	For each $n \geq 3$,
	there exists a compact simply-connected $(2 n + 1)$-manifold $X$
	admitting two strongly pseudoconvex CR structures $(H, J)$ and $(\wtH, \wtJ)$
	such that they have different signs of the CR Yamabe constants,
	and $(M, H)$ and $(M, \wtH)$ are not isomorphic as cooriented contact manifolds.
\end{theorem}

We remark that the existence of CR structures with different signs of the CR Yamabe constant
on a fixed contact structure remains unresolved.

This paper is organized as follows.
In \cref{section:CR-manifolds},
we recall basic facts on CR manifolds
and show that a principal $S^{1}$-bundle over a Hodge manifold has a canonical strongly pseudoconvex CR structure.
Some integral formulae of the CR Yamabe constant are given in \cref{section:formulae-for-CR-Yamabe-constant}.
In \cref{section:continuity-of-the-CR-Yamabe-constant},
we prove the continuity of the CR Yamabe constant
under suitable deformations of CR structures,
which will be used for the proof of \cref{thm:varying-CR-Yamabe-constant}.
\cref{section:deformations-of-CR-structures} is devoted to
constructions of deformations of strongly pseudoconvex CR structures
with varying CR Yamabe constants.
In \cref{section:construction-of-manifold-with-inequivalent-CR-structures},
we give a proof of \cref{thm:inequivalent-CR-manifold}.

\section{CR manifolds}
\label{section:CR-manifolds}

An almost CR structure on a smooth $(2 n + 1)$-manifold $X$ is a pair $(H, J)$
where $H \subset T X$ is a codimension $1$ smooth subbundle with an almost complex structure $J$.
An almost CR structure is called integrable or a CR structure if
\begin{equation}
	\comm{\Gamma(T^{1, 0} X)}{\Gamma(T^{1, 0} X)} \subset \Gamma(T^{1, 0} X)
\end{equation}
for
\begin{equation}
	T^{1, 0} X
	\coloneqq \Set{v - i J v | v \in H} \subset H \otimes \bbC
	\qquad
	T^{0, 1} X
	\coloneqq \overline{T^{1, 0} X}.
\end{equation}

We shall consider only an orientable CR manifold.
Then one can choose a smooth real-valued $1$-form $\theta$ annihilating exactly $H$,
which is determined up to multiplication by a nowhere vanishing real-valued function on $X$.
By the integrability condition,
$d \theta$ is $J$-invariant;
i.e., a $(1, 1)$-form,
and hence one can introduce the symmetric bilinear form
\begin{equation}
	L_{\theta}
	\coloneqq d \theta(\cdot, J \cdot)
\end{equation}
defined on $H$,
called the \emph{Levi form}.

A CR manifold $(X, H, J)$ is called \emph{strongly pseudoconvex}
if $L_\theta$ is definite for some (and hence all) $\theta$,
and a strongly pseudoconvex CR manifold $(X, H, J)$ with a choice of $\theta$
is called a \emph{pseudohermitian} manifold.
We shall always assume that $X$ is strongly pseudoconvex
and $\theta$ is chosen so that $L_{\theta}$ is positive definite,
unless otherwise specified.
In this case,
the distribution $H$ is a contact structure on $X$ with a contact form $\theta$.
Let $T$ be its Reeb vector field;
i.e., the unique vector field satisfying $\theta(T) = 1$ and $\iota(T) d \theta = 0$.
The Levi form induces a Hermitian metric $L_{\theta}^{\ast}$ on $H^{\ast}$.
The \emph{sublaplacian} $\Delta_{b}$ is defined by
\begin{equation}
	\int_{X} (\Delta_{b} u) v \, d \mu_{\theta}
	= \int_{X} L_{\theta}^{\ast}(d u|_{H}, d v|_{H}) d \mu_{\theta}.
\end{equation}

A set of local $1$-forms $\{\theta^{1}, \cdots, \theta^{n}\}$ of type $(1, 0)$ is called \emph{admissible},
if its restriction to $T^{1, 0} X$ forms a basis of $(T^{1, 0} X)^*$ at each point
and $\theta^{\alpha}(T) = 0$ for all $\alpha$.
For an admissible coframe,
we have
\begin{equation}
\label{eq:Levi-form}
	d \theta
	= i h_{\alpha \ovxb} \theta^{\alpha} \wedge \theta^{\ovxb},
\end{equation}
where $(h_{\alpha \ovxb})$ is a positive-definite hermitian matrix of functions
and $\theta^{\ovxb} = \overline{\theta^{\beta}}$.
We shall always adopt the Einstein convention and use the matrix $(h_{\alpha \ovxb})$
and its inverse $(h^{\alpha \ovxb})$ to raise and lower indices.
The integrability condition of $J$ can be rephrased as
\begin{equation}
	d \theta^{\alpha}
	\equiv 0
	\qquad
	\textrm{mod} \ \theta, \theta^{\gamma}
\end{equation}
along with \eqref{eq:Levi-form}.

A pseudohermitian manifold carries a canonical linear connection,
the \emph{Tanaka-Webster connection} \cites{Tanaka1975,Webster1978},
whose connection 1-forms ${\omega_{\alpha}}^{\beta}$
and torsion forms $\tau^{\alpha}$ of type $(0,1)$ are uniquely determined by the following relations
\begin{equation}
	d \theta^{\beta}
	= \theta^{\alpha} \wedge {\omega_{\alpha}}^{\beta} + \theta \wedge \tau^{\beta},
	\qquad
	\omega_{\alpha \ovxb} + \omega_{\ovxb \alpha}
	= d h_{\alpha \ovxb}
\end{equation}
together with \eqref{eq:Levi-form}.
We call $\tau^{\alpha}$ the \emph{pseudohermitian torsion}.
The whole torsion tensor is composed of $\theta \wedge \tau^{\gamma}$
and $i h_{\alpha \ovxb} \theta^{\alpha} \wedge \theta^{\ovxb}$,
and so it is nowhere vanishing.

The covariant differentiation with respect to this connection is given by
\begin{equation}
	\nabla Z_{\alpha}
	= {\omega_{\alpha}}^{\beta} \otimes Z_{\beta},
	\qquad
	\nabla Z_{\ovxa}
	= {\omega_{\ovxa}}^{\ovxb} \otimes Z_{\ovxb},
	\qquad
	\nabla T
	= 0,
\end{equation}
where a local frame $\{Z_{\alpha}\}$ of $T^{1,0}X$ is dual to $\{\theta^{\alpha}\}$.
Its curvature $2$-forms
\begin{equation}
	{\Omega_{\alpha}}^{\beta}
	\coloneqq d {\omega_{\alpha}}^{\beta} - {\omega_{\alpha}}^{\gamma} \wedge {\omega_{\gamma}}^{\beta}
\end{equation}
may have several types,
but one considers only its $(1, 1)$-part
\begin{equation}
	{{R_{\alpha}}^{\beta}}_{\rho \ovxs} \theta^{\rho} \wedge \theta^{\ovxs}
\end{equation}
to get its \emph{pseudohermitian Ricci tensor}
\begin{equation}
	R_{\rho \ovxs}
	\coloneqq {{R_{\alpha}}^{\alpha}}_{\rho \ovxs}
\end{equation}
by taking its contraction.
Finally its \emph{pseudohermitian scalar curvature} is the metric contraction
${R_{\rho}}^{\rho} = h^{\rho \ovxs} R_{\rho \ovxs}$.
We say $\theta$ to be \emph{Einstein}
if its pseudohermitian torsion is identically zero
and its pseudohermitian Ricci curvature is a constant multiple of the Levi form.

\begin{lemma}
\label{lem:Einstein-is-CR-Yamabe-minimizer}
	Any Einstein contact form is a CR Yamabe minimizer.
\end{lemma}

\begin{proof}
	Let $(X, H, J)$ be a compact strongly pseudoconvex CR manifold of dimension $2 n + 1$
	and $\theta$ be an Einstein contact form on $X$.
	If the pseudohermitian scalar curvature is non-positive,
	then it must be a CR Yamabe minimizer.
	If the pseudohermitian scalar curvature is positive,
	we may assume that it is equal to $n (n + 1)$ by homothety.
	Consider the Riemannian metric $g_{\theta}$ on $X$ given by
	\begin{equation}
		g_{\theta}(U, V)
		\coloneqq \frac{1}{2} d \theta(U, J V) + \theta(U) \theta(V),
		\qquad U, V \in T X.
	\end{equation}
	Here we extend $J$ to an endomorphism on $T M$ by $J T = 0$.
	Note that the volume form of $g_{\theta}$ coincides with $(2^{n} n!)^{- 1} \, d \mu_{\theta}$.
	This $g_{\theta}$ satisfies $\Ric_{g_{\theta}} = 2 n g_{\theta}$;
	see \cite{Takeuchi2018}*{Proposition 2.9} for example.
	The Bishop inequality implies that
	\begin{align}
		\mathfrak{F}(\theta)
		= n (n + 1) (2^{n} n ! \Vol_{g_{\theta}}(X))^{\frac{1}{n + 1}}
		&\leq n (n + 1) (2^{n} n ! \Vol_{g_{0}}(S^{2 n + 1}))^{\frac{1}{n + 1}} \\
		&= 2 n (n + 1) \pi,
	\end{align}
	where $g_{0}$ is the standard Riemannian metric on $S^{2 n + 1}$.
	Moreover,
	the equality holds if and only if $(X, g_{\theta})$ is isometric to $(S^{2 n + 1}, g_{0})$.
	In addition,
	if $(X, g_{\theta})$ is isometric to $(S^{2 n + 1}, g_{0})$,
	then $(X, H, J)$ is CR isomorphic to the standard CR sphere;
	see the paragraph after the proof of Proposition 4 in \cite{Wang2015}.
	Therefore $(X, H, J)$ has a CR Yamabe minimizer,
	and $\theta$ is also a CR Yamabe minimizer by \cite{Wang2015}*{Theorem 3}.
\end{proof}

An important example of a strongly pseudoconvex CR manifold
is a principal $S^{1}$-bundle over a Hodge manifold.
Given a Hodge manifold $(M, J, \omega)$;
that is,
its \Kahler class $[\omega]$ is an integral cohomology class,
we consider a principal $S^{1}$-bundle $p \colon P_{M} \to M$ whose Euler class is $- [\omega]$.
Recall that for any $\bbR$-valued principal connection $\theta$ on $P_{M}$,
$d \theta / 2 \pi$ descends to $M$ and its cohomology class coincides with $[\omega]$.
We take a principal connection $\theta$ satisfying $d \theta / 2 \pi = p^{\ast} \omega$,
and consider the lifted almost complex structure
\begin{equation}
	p^{\ast} J \colon H \coloneqq \Ker \theta \to H.
\end{equation}

\begin{proposition}
\label{prop:circle-bundle-over-Hodge-manifold}
	The triple $(P_{M}, H, p^{\ast} J)$ is a strongly pseudoconvex CR manifold.
	Moreover,
	the pseudohermitian scalar curvature of $(P_{M}, H, p^{\ast} J, \theta / 2 \pi)$
	is equal to $p^{\ast} S(\omega)$,
	where $S(\omega)$ is the scalar curvature of $(M, J, \omega)$.
\end{proposition}

\begin{proof}
	This result may be well-known for the researchers in Sasakian geometry,
	but we give a proof for the reader's convenience.
	Let $(z^{1}, \dots , z^{n})$ be a holomorphic local coordinate of $M$.
	Then $\theta^{\alpha} \coloneqq p^{\ast} d z^{\alpha}$ defines
	an admissible coframe.
	Since $d \theta^{\alpha} = p^{\ast} d d z^{\alpha} = 0$,
	the almost CR structure $(H, p^{\ast} J)$ is integrable.
	Moreover,
	\begin{equation}
		L_{\theta}(X, X)
		= d \theta(X, (p^{\ast} J) X)
		= 2 \pi \omega(p_{\ast} X, J (p_{\ast} X))
		> 0
	\end{equation}
	for any non-zero $X \in H$,
	which implies that $(M, H, p^{\ast} J)$ is strongly pseudoconvex.

	Consider the Tanaka-Webster connection with respect to $\theta / 2 \pi$.
	The \Kahler form $\omega$ is written as
	\begin{equation}
		\omega
		= i g_{\alpha \ovxb} d z^{\alpha} \wedge d \ovz^{\beta},
	\end{equation}
	where $(g_{\alpha \ovxb})$ is a positive definite Hermitian matrix.
	Since $d \theta / 2 \pi = p^{*} \omega$,
	we have
	\begin{equation}
		d \theta / 2 \pi = i (p^{\ast} g_{\alpha \ovxb}) \theta^{\alpha} \wedge \theta^{\ovxb},
	\end{equation}
	which implies $h_{\alpha \ovxb} = p^{\ast} g_{\alpha \ovxb}$.
	The structure equation on the \Kahler manifold $(M, J, \omega)$ says that
	\begin{equation}
		0
		= d (d z^{\beta})
		= d z^{\alpha} \wedge {\phi_{\alpha}}^{\beta},
		\qquad
		d g_{\alpha \ovxb}
		= \phi_{\alpha \ovxb} + \phi_{\ovxb \alpha},
	\end{equation}
	where ${\phi_{\alpha}}^{\beta}$ is the Levi-Civita connection $1$-form,
	and so
	\begin{equation}
		d \theta^{\beta}
		= \theta^{\alpha} \wedge (p^{\ast} {\phi_{\alpha}}^{\beta}),
		\qquad
		d (p^{\ast} g_{\alpha \ovxb})
		= p^{\ast} \phi_{\alpha \ovxb} + p^{\ast} \phi_{\ovxb \alpha}.
	\end{equation}
	Hence the pull-back connection given by $p^{\ast} {\phi_{\alpha}}^{\beta}$
	coincides with the Tanaka-Webster connection.
	Note that the pseudohermitian torsion is equal to zero
	and $P_{M}$ is Sasakian.

	The curvature $2$-forms of the Levi-Civita connection on $(M, J, \omega)$ are given by
	\begin{equation}
		{\Phi_{\alpha}}^{\beta}
		\coloneqq d {\phi_{\alpha}}^{\beta} - {\phi_{\alpha}}^{\gamma} \wedge {\phi_{\gamma}}^{\beta}
		= {{R_{\alpha}}^{\beta}}_{\rho \ovxs} d z^{\rho} \wedge d z^{\ovxs}.
	\end{equation}
	Its pull-back to $P_{M}$ yields
	\begin{equation}
		p^{\ast} {\Phi_{\alpha}}^{\beta}
		= d (p^{\ast} {\phi_{\alpha}}^{\beta})
			- (p^{\ast} {\phi_{\alpha}}^{\gamma}) \wedge (p^{\ast} {\phi_{\gamma}}^{\beta})
		= (p^{\ast} {{R_{\alpha}}^{\beta}}_{\rho \ovxs}) \theta^{\rho} \wedge \theta^{\ovxs},
	\end{equation}
	which are the curvature $2$-forms of the Tanaka-Webster connection.
	Hence the pseudohermitian Ricci tensor $R_{\rho\bar{\sigma}}$ is given by
	\begin{equation}
		p^{\ast} {{R_{\alpha}}^{\alpha}}_{\rho \ovxs}
		= p^{\ast} \Ric_{\rho \ovxs},
	\end{equation}
	where $\Ric$ is the Ricci tensor of $(M, J, \omega)$,
	and the pseudohermitian scalar curvature is given by
	\begin{equation}
		{R_{\rho}}^{\rho}
		= p^{\ast} {\Ric_{\rho}}^{\rho}
		= p^{\ast} S(\omega),
	\end{equation}
	which completes the proof.
\end{proof}

\section{Formulae for CR Yamabe constant}
\label{section:formulae-for-CR-Yamabe-constant}

As with the Riemannian Yamabe problem,
the functional $\mathfrak{F}(\theta)$ can be rewritten as a functional on $C^{\infty}(X, \bbR_{+})$:
\begin{equation}
\label{eq:Yamabe-functional}
	\mathfrak{F}(u^{2 / n} \theta)
	= \frac{\int_{X} ((2 + 2 / n) \abs{d u}_{\theta}^{2} + R_{\theta} u^{2}) \, d \mu_{\theta}}
		{(\int_{X} u^{2 + 2 / n} \, d \mu_{\theta})^{n / (n + 1)}},
\end{equation}
where $\theta$ is a fixed contact form
and $\abs{d u}_{\theta}^{2} = L_{\theta}^{\ast} (d u|_{H}, d u|_{H})$~\cite{Jerison-Lee1987}.
It follows from this that $\lambda(X) \geq 0$ if $R_{\theta} \geq 0$.
Moreover consider the case $R_{\theta} > 0$.
Suppose to the contrary that $\lambda(X) = 0$.
Then there exists a CR Yamabe contact form $\wtxth = u^{2 / n} \theta$ satisfying $R_{\wtxth} = 0$,
which contradicts \eqref{eq:Yamabe-functional}.
Therefore $\lambda(X) > 0$ if $R_{\theta} > 0$;
see \cite{Wang2003}*{Proposition 3.1} for another proof.

\begin{theorem}
\label{thm:integral-formulae-of-CR-Yamabe-constant}
	Let $(X, H, J)$ be a compact strongly pseudoconvex CR manifold of dimension $2 n + 1$.
	If $\lambda(X) \geq 0$,
	then for any $r \in \clcl{1}{\infty}$
	\begin{equation}
		\lambda(X)
		\leq \inf_{\wtxth} \norm{R_{\wtxth}}_{L^{r}} \Vol_{\wtxth}(X)^{\frac{1}{n + 1} - \frac{1}{r}},
	\end{equation}
	where $\Vol_{\theta}(X)$ is the volume of $X$ with respect to $d \mu_{\theta}$.
	If the CR Yamabe problem is solvable on $X$, then the equality holds.
	If $\lambda(X) \leq 0$,
	then for any $r \in \clcl{n + 1}{\infty}$
	\begin{align}
		\lambda(X)
		&= - \inf_{\wtxth} \norm{R_{\wtxth}}_{L^{r}}
			\Vol_{\wtxth}(X)^{\frac{1}{n + 1} - \frac{1}{r}} \label{eq:take-1}\\
		&= - \inf_{\wtxth} \norm{R_{\wtxth}^{-}}_{L^{r}}
			\Vol_{\wtxth}(X)^{\frac{1}{n + 1} - \frac{1}{r}}, \label{eq:take-2}
	\end{align}
	where $R_{\wtxth}^{-} \coloneqq \min(R_{\wtxth}, 0)$,
	and the two infima are realized only by a CR Yamabe minimizer.
\end{theorem}

\begin{proof}
	When $\lambda(X) \geq 0$,
	the \Holder inequality implies
	\begin{equation}
		\lambda(X)
		\leq \frac{\int_{X} R_{\wtxth} \, d \mu_{\wtxth}}{\Vol_{\wtxth}(X)^{n / (n + 1)}}
		\leq \norm{R_{\wtxth}}_{L^{r}} \Vol_{\wtxth}(X)^{\frac{1}{n + 1} - \frac{1}{r}},
	\end{equation}
	and the equality holds if $\wtxth$ is a CR Yamabe minimizer.

	Now in the case of $\lambda(X) \leq 0$,
	we use the technique of Besson, Courtois, and Gallot~\cite{Besson-Courtois-Gallot1991}.
	Let $\theta$ be a CR Yamabe minimizer,
	which is unique up to a constant in this case.
	Consider another contact form $\wtxth = u^{2 / n} \theta$,
	where $u$ is a positive smooth function.
	The \Holder inequality yields
	\begin{align}
		\norm{R_{\wtxth}}_{L^{r}} \Vol_{\wtxth}(X)^{\frac{1}{n + 1} - \frac{1}{r}}
		&= \rbra*{\int_{X} \abs{R_{\wtxth}}^{r} \, d \mu_{\wtxth}}^{\frac{1}{r}}
			\Vol_{\wtxth}(X)^{\frac{1}{n + 1} - \frac{1}{r}} \\
		&\geq \rbra*{\int_{X} \abs{R_{\wtxth}^{-}}^{r} u^{2 + 2 / n} \, d \mu_{\theta}}^{\frac{1}{r}}
			\rbra*{\int_{X} u^{2 + 2 / n} \, d \mu_{\theta}}^{\frac{1}{n + 1} - \frac{1}{r}} \\
		&\geq \rbra*{\int_{X} (- R_{\wtxth}^{-}) u^{2 / n} \, d \mu_{\theta}} \Vol_{\theta}(X)^{- \frac{n}{n + 1}} \\
		&\geq \rbra*{\int_{X} (- R_{\wtxth}) u^{2 / n} \, d \mu_{\theta}} \Vol_{\theta}(X)^{- \frac{n}{n + 1}}.
	\end{align}
	Here recall that
	\begin{equation}
		R_{\wtxth}
		= u^{- 1 - 2 / n} (R_{\theta} + (2 + 2 / n) \Delta_{b}) u,
	\end{equation}
	where $\Delta_{b}$ is the sublaplacian~\cite{Jerison-Lee1987}.
	Therefore
	\begin{align}
		&\norm{R_{\wtxth}}_{L^{r}} \Vol_{\wtxth}(X)^{\frac{1}{n + 1} - \frac{1}{r}} \\
		&\geq \rbra*{\int_{X} (- R_{\theta} - (2 + 2 / n) u^{- 1} \Delta_{b} u) \, d \mu_{\theta}}
			\Vol_{\theta}(X)^{- \frac{n}{n + 1}} \\
		&= \rbra*{\int_{X} (- R_{\theta} + (2 + 2 / n) u^{- 2} \abs{d u}_{\theta}^{2}) \, d \mu_{\theta}}
			\Vol_{\theta}(X)^{- \frac{n}{n + 1}} \\
		&\geq - \rbra*{\int_{X} R_{\theta} \, d \mu_{\theta}} \Vol_{\theta}(X)^{- \frac{n}{n + 1}} \\
		&= - \lambda(X).
	\end{align}
	This proves
	\begin{equation}
		\lambda(X)
		\geq - \inf_{\wtxth} \norm{R_{\wtxth}^{-}}_{L^{r}}
			\Vol_{\wtxth}(X)^{\frac{1}{n + 1} - \frac{1}{r}}
		\geq - \inf_{\wtxth} \norm{R_{\wtxth}}_{L^{r}}
			\Vol_{\wtxth}(X)^{\frac{1}{n + 1} - \frac{1}{r}}.
	\end{equation}
	On the other hand,
	\begin{equation}
		\lambda(X)
		= - \norm{R_{\theta}^{-}}_{L^{r}}
			\Vol_{\theta}(X)^{\frac{1}{n + 1} - \frac{1}{r}}
		= - \norm{R_{\theta}}_{L^{r}}
			\Vol_{\theta}(X)^{\frac{1}{n + 1} - \frac{1}{r}}
	\end{equation}
	since $R_{\theta}$ is non-positive constant,
	which proves the two desired formulae together.
	It remains to decide when the infima are realized.
	The infimum of \eqref{eq:take-1} or \eqref{eq:take-2} is realized by $\wtxth$
	if and only if the above all inequalities are attained by equalities,
	which holds if and only if $R_{\wtxth}^{-} \equiv R_{\wtxth}$ is a non-positive constant
	(and hence $u$ is a positive constant);
	i.e., $\wtxth$ is a CR Yamabe minimizer.
\end{proof}

\section{Continuity of the CR Yamabe constant}
\label{section:continuity-of-the-CR-Yamabe-constant}

In this section,
we prove the continuity of the CR Yamabe constant
under suitable deformations of CR structures.
Remark that \cref{lem:continuity-of-CRYc} below is a generalization of \cite{Dietrich2021}*{Lemma 5.5}.

\begin{proposition}
\label{prop:continuity-of-CRYc}
	Let $(X, H, J, \theta)$ be a compact pseudohermitian manifold of dimension $2 n + 1$.
	Assume that $(X, H_{i}, J_{i}, \theta_{i})_{i \in \bbN}$
	is a sequence of pseudohermitian structures on $X$
	such that $\theta_{i} \to \theta$ in $C^{2}$-topology,
	and $J_{i} \to J$ and $R_{\theta_{i}} \to R_{\theta}$ in $C^{0}$-topology,
	where $J_{i}$ and $J$ extend to endomorphisms on $T X$ in an obvious way.
	Then one has $\lambda(X, H_{i}, J_{i}) \to \lambda(X, H, J)$.
\end{proposition}

\begin{proof}
	Since $\theta_{i} \to \theta$ in $C^{2}$-topology,
	we may assume that $\theta_{i}^{t} \coloneqq t \theta_{i} + (1 - t) \theta$ for $t \in \clcl{0}{1}$
	is a smooth family of contact forms on $X$.
	Let $T_{i}^{t}$ be the Reeb vector field of $\theta_{i}^{t}$,
	which is determined by
	\begin{equation}
		\theta_{i}^{t}(T_{i}^{t}) = 1,
		\qquad
		\iota(T_{i}^{t}) d \theta = 0.
	\end{equation}
	Since $\theta_{i} \to \theta$ in $C^{2}$-topology,
	$\sup_{t \in \clcl{0}{1}} \norm{T_{i}^{t} - T}_{C^{1}} \to 0$.
	Take the time-dependent vector field $V_{i}^{t} \in \Gamma(\Ker \theta_{i}^{t})$ satisfying
	\begin{equation}
		\iota(V_{i}^{t}) d \theta_{i}^{t}
		= (\theta_{i} - \theta)(T_{i}^{t}) \theta_{i}^{t} - (\theta_{i} - \theta).
	\end{equation}
	It follows from $\sup_{t \in \clcl{0}{1}} \norm{\theta_{i}^{t} - \theta}_{C^{2}} \to 0$
	and $\sup_{t \in \clcl{0}{1}} \norm{T_{i}^{t} - T}_{C^{1}} \to 0$ that
	\begin{equation}
	\label{eq:norm-of-tdvf}
		\sup_{t \in \clcl{0}{1}} \norm{V_{i}^{t}}_{C^{1}} \to 0.
	\end{equation}
	The isotopy $\psi_{i}^{t} \colon X \to X$ generated by $V_{i}^{t}$ satisfies
	$(\psi_{i}^{t})^{\ast} H_{i}^{t} = H$ for any $t \in \clcl{0}{1}$;
	see the proof of the Gray stability theorem~\cite{Geiges2008}*{Theorem 2.2.2}.
	The equation \eqref{eq:norm-of-tdvf} yields that
	$\psi_{i}^{1} \to \id_{X}$ in $C^{1}$-topology;
	see e.g.\ \cite{Zhang2022} for a modern treatment of time-dependent vector fields with parameters
	and a proof of this fact.
	In particular,
	\begin{gather}
		\wtxth_{i} \coloneqq (\psi_{i}^{1})^{\ast} \theta_{i} \to \theta,
		\qquad
		d \wtxth_{i} = (\psi_{i}^{1})^{\ast} d \theta_{i} \to d \theta, \\
		\wtJ_{i} \coloneqq (\psi_{i}^{1})^{\ast} J_{i} \to J,
		\qquad
		R_{\wtxth_{i}} = (\psi_{i}^{1})^{\ast} R_{\theta_{i}} \to R_{\theta}
	\end{gather}
	in $C^{0}$-topology.
	Since $\lambda(X, H_{i}, J_{i}) = \lambda(X, H, \wtJ_{i})$,
	the statement follows from the Lemma below.
\end{proof}

\begin{lemma}
\label{lem:continuity-of-CRYc}
	Let $(X, H, J, \theta)$ be a compact pseudohermitian manifold of dimension $2 n + 1$.
	Assume that $(X, H, J_{i}, \theta_{i})_{i \in \bbN}$
	is a sequence of pseudohermitian structures on $X$
	such that $\theta_{i} \to \theta$, $d \theta_{i} \to d \theta$,
	$J_{i} \to J$, and $R_{\theta_{i}} \to R_{\theta}$ in $C^{0}$-topology.
	Then one has $\lambda(X, H, J_{i}) \to \lambda(X, H, J)$.
\end{lemma}

\begin{proof}
	Without loss of generality,
	we may assume that $\Vol_{\theta_{i}}(X) = \Vol_{\theta}(X) = 1$.
	Since $R_{\theta_{i}} \to R_{\theta}$ in $C^{0}$-topology,
	we can find $K > 0$ such that $\abs{R_{\theta}} < K$ and $\abs{R_{\theta_{i}}} < K$ for any $i$.
	For each $\varepsilon \in \opop{0}{1}$,
	take $N(\varepsilon) \in \bbZ_{+}$ such that
	$i \geq N(\varepsilon)$ implies
	\begin{gather}
		(1 + \varepsilon)^{- 1} L_{\theta} < L_{\theta_{i}} < (1 + \varepsilon) L_{\theta}, \\
		(1 + \varepsilon)^{- 1} d \mu_{\theta} < d \mu_{\theta_{i}} < (1 + \varepsilon) d \mu_{\theta}, \\
		\abs{R_{\theta_{i}} - R_{\theta}} < \varepsilon,
	\end{gather}
	For any $f \in C^{\infty}(X, \bbR)$,
	we write $f^{+} \coloneqq \max(f, 0)$ and $f^{-} \coloneqq \min(f, 0)$.
	Let $u \in C^{\infty}(X, \bbR_{+})$.
	The \Holder inequality yields
	\begin{equation}
		\int_{X} u^{2} \, d \mu_{\theta}
		\leq \rbra*{\int_{X} u^{2 + 2 / n} \, d \mu_{\theta}}^{\frac{n}{n + 1}} \Vol_{\theta}(X)^{\frac{1}{n + 1}}
		= \rbra*{\int_{X} u^{2 + 2 / n} \, d \mu_{\theta}}^{\frac{n}{n + 1}}
	\end{equation}
	and
	\begin{equation}
		\abs*{\int_{X} R_{\theta}^{\pm} u^{2} \, d \mu_{\theta}}
		\leq K \int_{X} u^{2} \, d \mu_{\theta}
		\leq K \rbra*{\int_{X} u^{2 + 2 / n} \, d \mu_{\theta}}^{\frac{n}{n + 1}}.
	\end{equation}
	It follows from the above inequalities that
	\begin{align}
		\int_{X} R_{\theta_{i}} u^{2} \, d \mu_{\theta_{i}}
		&\leq \int_{X} (R_{\theta} + \varepsilon) u^{2} \, d \mu_{\theta_{i}} \\
		&\leq \int_{X} (R_{\theta}^{+} + \varepsilon) u^{2} \, d \mu_{\theta_{i}}
			+ \int_{X} R_{\theta}^{-} u^{2} \, d \mu_{\theta_{i}} \\
		&\leq (1 + \varepsilon) \int_{X} (R_{\theta}^{+} + \varepsilon) u^{2} \, d \mu_{\theta}
			+ (1 + \varepsilon)^{- 1} \int_{X} R_{\theta}^{-} u^{2} \, d \mu_{\theta},
	\end{align}
	and
	\begin{equation}
		(1 + \varepsilon)^{- 1} \int_{X} u^{2 + 2 / n} d \mu_{\theta}
		\leq \int_{X} u^{2 + 2 / n} d \mu_{\theta_{i}}
		\leq (1 + \varepsilon) \int_{X} u^{2 + 2 / n} d \mu_{\theta}.
	\end{equation}
	Thus we have
	\begin{align}
		&\mathfrak{F}(u^{2 / n} \theta_{i}) \\
		&= \sbra*{\int_{X} ((2 + 2 / n) \abs{d u}_{\theta_{i}}^{2} + R_{\theta_{i}} u^{2}) \, d \mu_{\theta_{i}}}
			\rbra*{\int_{X} u^{2 + 2 / n} \, d \mu_{\theta_{i}}}^{- \frac{n}{n + 1}} \\
		&\leq \bigg[(1 + \varepsilon)^{2 + n / (n + 1)} \int_{X} (2 + 2 / n) \abs{d u}_{\theta}^{2} \, d \mu_{\theta}
			+ (1 + \varepsilon)^{1 + n / (n + 1)} \int_{X} (R_{\theta}^{+} + \varepsilon) u^{2} \, d \mu_{\theta} \\
		&\quad + (1 + \varepsilon)^{- 1 - n / (n + 1)} \int_{X} R_{\theta}^{-} u^{2} \, d \mu_{\theta} \bigg]
			\rbra*{\int_{X} u^{2 + 2 / n} \, d \mu_{\theta}}^{- \frac{n}{n + 1}} \\
		&\leq (1 + \varepsilon)^{2 + n / (n + 1)} \mathfrak{F}(u^{2 / n} \theta) + C \varepsilon,
	\end{align}
	where $C$ is a positive constant independent of $u$ and $\varepsilon$.
	Taking the infimum yields
	\begin{equation}
		\lambda(X, H, J_{i})
		\leq (1 + \varepsilon)^{2 + n / (n + 1)} \lambda(X, H, J) + C \varepsilon.
	\end{equation}
	Since $(J, \theta)$ and $(J_{i}, \theta_{i})$ are symmetric,
	we also obtain
	\begin{equation}
		\lambda(X, H, J)
		\leq (1 + \varepsilon)^{2 + n / (n + 1)} \lambda(X, H, J_{i}) + C \varepsilon.
	\end{equation}
	Since $\varepsilon > 0$ is arbitrary,
	we have
	\begin{equation}
		\limsup_{i \to \infty} \lambda(X, H, J_{i})
		\leq \lambda(X, H, J)
		\leq \liminf_{i \to \infty} \lambda(X, H, J_{i}),
	\end{equation}
	which implies $\lambda(X, H, J_{i}) \to \lambda(X, H, J)$.
\end{proof}

\section{Deformations of CR structures with varying CR Yamabe constants}
\label{section:deformations-of-CR-structures}

In this section,
we construct a family of strongly pseudoconvex CR structures
with varying CR Yamabe constants.
Let $(M, J, \omega)$ be an $n$-dimensional compact Hodge manifold with constant scalar curvature.
Let us write
\begin{equation}
	\calK
	\coloneqq \Set{\varphi \in C^{\infty}(M) | \omega_{\varphi} \coloneqq \omega + i \del \delb \varphi > 0}
\end{equation}
endowed with the $C^{4}$-topology
for the space of \Kahler potentials in the class $[\omega]$.
For any $\varphi \in \calK$,
we have
\begin{equation}
	\int_{M} \, \omega_{\varphi}^{n}
	= [\omega]^{n},
	\qquad
	\int_{M} S(\omega_{\varphi}) \, \omega_{\varphi}^{n}
	= 2 n \pi c_{1}(M) \cup [\omega]^{n - 1},
\end{equation}
which are independent of the choice of $\varphi$.
In particular if $\omega_{\varphi}$ is a constant scalar curvature \Kahler metric,
then
\begin{equation}
	S(\omega_{\varphi})
	= \frac{2 n \pi c_{1}(M) \cup [\omega]^{n - 1}}{[\omega]^{n}}
	\eqqcolon \whS.
\end{equation}
Set
\begin{equation}
	\calF
	\coloneqq \Set{\varphi \in \calK | S(\omega_{\varphi}) = \whS}
\end{equation}
so that $\omega_{\varphi}$ is not a constant scalar curvature \Kahler metric
for any $\varphi \in \calK \setminus \calF$.
It is known that $\omega_{\varphi}$ is a constant scalar curvature \Kahler metric if and only if
there exists $F \in \Aut^{0}(M)$
such that $\omega_{\varphi} = F^{\ast} \omega$~\cite{Berman-Berndtsson2017}*{Theorem 1.3}.
In particular if $\Aut(M)$ is discrete,
or equivalently,
$M$ admits no non-trivial holomorphic vector fields,
then $\calF = \bbR$.
More generally,
if any holomorphic vector field is parallel,
then $\calF = \bbR$.
This is because a parallel and holomorphic vector field preserves the \Kahler form $\omega$.

For each $\varphi \in \calK$,
take a principal connection $\theta_{\varphi}$ on $P_{M}$
satisfying $d \theta_{\varphi} / 2 \pi = p^{\ast} \omega_{\varphi}$,
which is given by
\begin{equation}
	\theta_{\varphi}
	\coloneqq \theta + \pi p^{\ast} d^{c} \varphi,
\end{equation}
where $d^{c} \coloneqq i(\delb - \del)$.
This gives an infinite-dimensional family of pseudohermitian manifolds
\begin{equation}
	X_{\varphi}
	\coloneqq (P_{M}, H_{\varphi} \coloneqq \Ker \theta_{\varphi}, p^{\ast} J, \theta_{\varphi} / 2 \pi)
\end{equation}
underlying the same manifold $P_{M}$.
Integration along fibers yields that
\begin{align}
	\mathfrak{F}(\theta_{\varphi} / 2 \pi)
	= \frac{\int_{P_{M}} p^{\ast} S(\omega_{\varphi})
		\, (\theta_{\varphi} / 2 \pi) \wedge (p^{\ast} \omega_{\varphi})^{n}}
		{(\int_{P_{M}} \, (\theta_{\varphi} / 2 \pi)\wedge (p^{\ast} \omega_{\varphi})^{n})^{n / (n + 1)}}
	&= \frac{\int_{M} S(\omega_{\varphi}) \, \omega_{\varphi}^{n}}
		{(\int_{M} \, \omega_{\varphi}^{n})^{n / (n + 1)}} \\
	&= \frac{2 n \pi c_{1}(M) \cup [\omega]^{n - 1}}{([\omega]^{n})^{n / (n + 1)}},
\end{align}
which is independent of $\varphi$.

\begin{proof}[Proof of \cref{thm:varying-CR-Yamabe-constant}]
	The continuity follows from \cref{prop:continuity-of-CRYc}.
	If $\theta_{0} = \theta$ is a CR Yamabe minimizer,
	then $\lambda(P_{M}, H, p^{\ast} J) = \mathfrak{F}(\theta_{0} / 2 \pi)$.
	On the other hand,
	$p^{\ast} S(\omega_{\varphi})$ is not constant for any $\varphi \in \calU \setminus \calF$,
	and so
	\begin{equation}
		\lambda(P_{M}, H_{\varphi}, p^{\ast} J)
		< \mathfrak{F}(\theta_{\varphi} / 2 \pi)
		= \mathfrak{F}(\theta_{0} / 2 \pi)
		= \lambda(P_{M}, H, p^{\ast} J),
	\end{equation}
	which completes the first assertion.
	If $\omega$ has constant non-positive scalar curvature,
	then so does $\theta / 2 \pi$,
	which must be a CR Yamabe minimizer.
	If $\omega$ is a \Kahler-Einstein metric,
	then $\theta / 2 \pi$ is an Einstein contact form.
	It follows from \cref{lem:Einstein-is-CR-Yamabe-minimizer} that $\theta / 2 \pi$ is a CR Yamabe minimizer.
\end{proof}

In general,
it may be cumbersome to have non-trivial $\calF$,
since it can be singular and not easy to locate.
In fact,
we do not have many examples with non-trivial $\calF$.
We give some examples of Hodge manifolds with $\calF = \bbR$.

\begin{example}[\Kahler-Einstein manifolds with non-positive scalar curvature]
	First,
	a compact complex manifold $M$ with $c_{1}(M) < 0$ admits a \Kahler-Einstein metric
	in the \Kahler class $- c_{1}(M)$ by the Aubin-Yau theorem~\cites{Aubin1976, Yau1978}.
	Second,
	the celebrated Calabi-Yau theorem \cite{Yau1978} implies that
	any \Kahler class on a compact complex manifold $M$ with $c_{1}(M) = 0$ in $H^{2}(M; \bbR)$
	is represented by a unique Ricci-flat \Kahler metric.
	Thus we can take any integral \Kahler class for our purpose.
	In these cases,
	a constant scalar curvature \Kahler metric in any \Kahler class,
	if it exists,
	must be unique~\cite{Chen2000}*{Theorem 7}
	and hence $\calF = \bbR$.
\end{example}

\begin{example}[Fano manifolds]
	Let $M$ be a Fano manifold;
	that is,
	a compact complex manifold with $c_{1}(M) > 0$.
	If $M$ admits a \Kahler-Einstein metric $\omega$ in the \Kahler class $c_{1}(M)$,
	then $\theta / 2 \pi$ gives a CR Yamabe minimizer.
	For example,
	a complex surface given by a blow-up of $\cps^{2}$ at $m$ points
	in general position with $3 \leq m \leq 8$
	admits a \Kahler-Einstein metric of positive scalar curvature~\cites{Tian-Yau1987, Tian1990}.
	Note that the automorphism groups of these surfaces are discrete.
	As a higher-dimensional example,
	consider the Fermat hypersurface $F_{n, d} \subset \cps^{n + 1}$ of degree $3 \leq d \leq n + 1$;
	that is;
	\begin{equation}
		F_{n, d}
		\coloneqq \Set{[z_{0}: \dots: z_{n + 1}] \in \cps^{n + 1} | \sum_{k = 0}^{n + 1} z_{k}^{d} = 0}.
	\end{equation}
	This $F_{n, d}$ admits a \Kahler-Einstein metric~\cite{Tian2000}*{Section 6.3}.
	Note that $F_{n, d}$ has no non-trivial holomorphic vector field~\cite{Kodaira-Spencer1958}*{Lemma 14.2};
	see \cite{Matsumura-Monsky1964} for another proof.
\end{example}

\begin{example}[Scalar-flat but not Ricci-flat surfaces]
	Take a complex surface $S_{m}$ obtained by blowing up $\cps^{1} \times \Sigma$
	at generic $m$ points $p_{1}, \cdots, p_{m}$,
	where $\Sigma$ is a compact Riemann surface of genus $g \geq 2$ and $m \geq 3$.
	It is known that $S_{m}$ admits a scalar-flat \Kahler form $\wtxo$~\cite{LeBrun-Singer1993}*{Theorem 3.11}.
	Note that $S_{m}$ does not admit a Ricci-flat \Kahler metric
	since $c_{1}(S_{m})^{2} = 8 (1 - g) - m < 0$.
	Assume that the projection of $\{p_{1}, \cdots, p_{m}\}$ to $\cps^{1}$ consists of at least $3$ points.
	Since $\Sigma$ has no non-trivial holomorphic vector fields
	and any holomorphic vector field on $\cps^{1}$ vanishing at least $3$ points must be trivial,
	$S_{m}$ admits no non-trivial holomorphic vector fields also.

	It still remains to show that $S_{m}$ has a scalar-flat integral \Kahler class.
	Since $H^{2, 0}(S_{m}) = 0$,
	we have $H^{1, 1}(S_{m}; \bbR) \cong H^{2}(S_{m}; \bbZ) \otimes \bbR$.
	The linear operator
	\begin{equation}
		c_{1}(S_{m}) \cup \colon H^{1, 1}(S_{m}; \bbR) \to \bbR
	\end{equation}
	has integer coefficients
	and satisfies $c_{1}(S_{m}) \cup [\wtxo] = 0$.
	Hence there exists a rational \Kahler class $\mu$ close enough to $[\wtxo]$
	satisfying $c_{1}(S_{m}) \cup \mu = 0$.
	\cite{LeBrun-Simanca1995}*{Theorem 1} implies that
	$\mu$ contains a scalar-flat \Kahler form $\omega$.
	Therefore we obtain a scalar-flat integral \Kahler class by homothety.
\end{example}

\section{Construction of a manifold with different signs of CR Yamabe constants}
\label{section:construction-of-manifold-with-inequivalent-CR-structures}

In this section,
we construct a manifold admitting two strongly pseudoconvex CR structures
with different signs of CR Yamabe constants.
Our construction is based on a exotic smooth structure of a certain complex surface
and an adaptation of the technique originated by Ruan~\cite{Ruan1994} and used by Kim and Sung~\cite{Kim-Sung2016}
to show the existence of inequivalent symplectic structures on certain $6$-manifolds.

Let $B$ be the Barlow surface~\cites{Barlow1985,Kotschick1989}
and $R_{8}$ be a complex surface given by a blow-up of $\cps^{2}$ at $8$ points in general position.
The Barlow surface is a simply-connected minimal surface of general type with $q = p_{g} = 0$ and $c_{1}(B)^2 = 1$,
and contains $(- 2)$-curves so that its canonical line bundle is not ample.
But as shown in \cite{Catanese-LeBrun1997}*{Theorem 7},
it has a small deformation with ample canonical line bundle
and hence admits a \Kahler-Einstein metric of negative scalar curvature by the celebrated Aubin-Yau theorem.
By the results of Tian and Yau \cites{Tian-Yau1987, Tian1990},
$R_{8}$ admits a \Kahler-Einstein metric of positive scalar curvature.
It is well-known that $B$ and $R_{8}$ are homeomorphic by Freedman's classification~\cite{Freedman1982}*{Theorem 1.5}
while they are not diffeomorphic by Kotschick's theorem~\cite{Kotschick1989}*{Theorem 1}.

\begin{remark}
	An easier way of proving Kotschick's theorem by using Seiberg-Witten invariant runs as follows.
	Since $R_{8}$ and $B$ have $b_{2}^{+} = 1$,
	their Seiberg-Witten invariants for a $\mathrm{Spin}^{c}$ structure $\xi$ with $c_{1}(\xi)^{2} > 0$
	are well-defined for any small perturbation.
	The complex surface $R_{8}$ admits a metric of positive scalar curvature,
	so its Seiberg-Witten invariants all vanish.
	However the Seiberg-Witten invariant of $B$
	for the canonical $\mathrm{Spin}^c$ structure determined by the complex structure is $\pm 1$~\cite{Morgan1996}.
\end{remark}

Since the intersection forms of both $B$ and $R_{8}$ are indefinite and odd,
they are isomorphic to $(1) \oplus 8(- 1)$.
Wall~\cite{Wall1962}*{p.336} has proved that
all characteristic vectors with square $1$ in $(1) \oplus 8(- 1)$ are equivalent.
Since the first Chern class of $B$ and $R_{8}$ are characteristic with square $1$ by Wu's formula,
there is an isomorphism from $H^{2}(R_{8}; \bbZ)$ to $H^{2}(B; \bbZ)$
preserving the intersection form and the first Chern class.
This induces an isomorphism
\begin{equation}
	\Psi \colon H^{\ast}(R_{8} \times \cps^{1}; \bbZ) \to H^{\ast}(B \times \cps^{1}; \bbZ)
\end{equation}
in the obvious way preserving $H^{\ast}(\cps^{1}; \bbZ)$.
We claim that $\Psi$ satisfies the conditions of the following theorem.

\begin{theorem}[\cite{Jupp1973}*{Theorem 1}]
	Let $X$ and $Y$ be smooth closed simply-connected $6$-manifolds with torsion-free homology.
	Suppose that there is an isomorphism from $H^{\ast}(X; \bbZ)$ to $H^{\ast}(Y; \bbZ)$
	preserving the triple cup product structure $\mu \colon H^{2} \otimes H^{2} \otimes H^{2} \to \bbZ$,
	the second Stiefel-Whitney class,
	and the first Pontryagin class.
	Then there exists an orientation-preserving diffeomorphism from $X$ to $Y$
	realizing this algebraic isomorphism.
\end{theorem}

It is enough to check that $\Psi$ preserves the specified characteristic classes.
By the product formula,
\begin{align}
	w_{2}(R_{8} \times \cps^{1})
	&= w_{2}(R_{8}) + w_{1}(R_{8}) w_{1}(\cps^{1}) + w_{2}(\cps^{1}) \\
	&= c_{1}(R_{8}) + 0 + c_{1}(\cps^{1}) \qquad \textrm{mod} \ 2,
\end{align}
and likewise for $B \times \cps^{1}$.
Since
\begin{equation}
	\Psi(c_{1}(R_{8}) + c_{1}(\cps^{1}))
	= c_{1}(B) + c_{1}(\cps^{1}),
\end{equation}
$\Psi$ preserves the second Stiefel-Whitney class.
Using the fact that $p_{1} = c_{1}^{2} - 2 c_{2}$ and the product formula,
we have
\begin{align}
	p_{1}(R_{8} \times \cps^{1})
	&= c_{1}(R_{8} \times \cps^{1})^{2} - 2 c_{2}(R_{8} \times \cps^{1}) \\
	&= (c_{1}(R_{8}) + c_{1}(\cps^{1}))^{2} - 2 (c_{2}(R_{8}) + c_{1}(R_{8}) c_{1}(\cps^{1})),
\end{align}
and likewise for $B \times \cps^{1}$.
Since $\Psi$ preserves the Euler characteristic;
i.e., the alternating sum of Betti numbers,
$\Psi$ maps $e(R_{8}) = c_{2}(R_{8})$ to $e(B) = c_{2}(B)$.
Therefore $\Psi$ preserves the first Pontryagin class too,
and we have an orientation-preserving diffeomorphism
\begin{equation}
	\psi \colon R_{8} \times \cps^{1} \to B \times \cps^{1}
\end{equation}
satisfying $\psi^{\ast}(c_{1}(B)) = c_{1}(R_{8})$ and $\psi^{\ast}(c_{1}(\cps^{1})) = c_{1}(\cps^{1})$.

Let $n \geq 3$.
Take \Kahler forms $\omega_{1} \in c_{1}(R_{8})$,
$\omega_{2} \in c_{1}(\cps^{1})$,
and $\omega_{3} \in n c_{1}(\cps^{n - 3})$ such that
\begin{equation}
	\Ric(\omega_{1})
	= 2 \pi \omega_{1},
	\qquad
	\Ric(\omega_{2})
	= 2 \pi \omega_{2},
	\qquad
	\Ric(\omega_{3})
	= \frac{2 \pi}{n} \omega_{3}.
\end{equation}
Then
\begin{equation}
	(M \coloneqq R_{8} \times \cps^{1} \times \cps^{n - 3}, \omega \coloneqq \omega_{1} + \omega_{2} + \omega_{3})
\end{equation}
is a \Kahler manifold with positive Ricci curvature.
Let $P_{M}$ be the principal $S^{1}$-bundle over $M$ whose Euler class is $- [\omega]$.
This $P_{M}$ admits a connection one-form $\theta$ and the lifted CR structure $J$
such that $d \theta / 2 \pi$ projects down to $\omega$ on $M$.
The CR Yamabe constant of $(P_{M}, H \coloneqq \Ker \theta, J)$ must be positive
by the argument after formula \eqref{eq:Yamabe-functional}.

\begin{lemma}
	The manifold $P_{M}$ is simply-connected.
\end{lemma}

\begin{proof}
	Let $E \cong \cps^{1}$ be an exceptional divisor in $R_{8}$
	and consider it as a complex curve in $M$.
	Then the restriction of $- [\omega]$ to $E$
	coincides with the first Chern class of the tautological line bundle over $E$.
	Hence $(P_{M})|_{E} \to E$ is a Hopf fibration.
	Consider the following commutative diagram with exact rows.
	\begin{equation}
	\begin{tikzcd}
		\pi_{2}((P_{M})|_{E}) \arrow[r] & \pi_{2}(E) \arrow[r] \arrow[d] & \pi_{1}(S^{1}) \arrow[r] \arrow[d]
			& \pi_{1}((P_{M})|_{E}) \arrow[d] & \\
		 & \pi_{2}(M) \arrow[r] & \pi_{1}(S^{1}) \arrow[r] & \pi_{1}(P_{M}) \arrow[r] & \pi_{1}(M) = 0
	\end{tikzcd}
	\end{equation}
	Since $(P_{M})|_{E} \to E$ is a Hopf fibration,
	$\pi_{1}((P_{M})|_{E}) = 0$ and $\pi_{2}((P_{M})|_{E}) = 0$,
	and so the map $\pi_{2}(E) \to \pi_{1}(S^{1})$ is an isomorphism.
	Thus we have $\pi_{2}(M) \to \pi_{1}(S^{1})$ is surjective
	and $\pi_{1}(P_{M}) = 0$.
\end{proof}

On the other hand,
let $\wtJ_{1}$ and $- \wtxo_{1}$ be the complex structure
and a \Kahler form in the class $- c_{1}(B)$ giving an Einstein metric of negative scalar curvature on $B$.
Denote by $\widetilde{M}^{\prime}$ the complex $3$-manifold
$(B \times \cps^{1}, (- \wtJ_{1}) \times J_{2})$.
The two-form $\wtxo^{\prime} = \wtxo_{1} + \omega_{2}$ defines
a \Kahler form on $\wtM^{\prime}$ with constant scalar curvature $- 2 \pi$.
Consider the \Kahler manifold
\begin{equation}
	(\widetilde{M} \coloneqq \widetilde{M}^{\prime} \times \cps^{n - 3},
	\widetilde{\omega} \coloneqq \widetilde{\omega}^{\prime} + \omega_{3}).
\end{equation}
The scalar curvature of this manifold is given by
$- 2 \pi + 2 (n - 3) \pi / n < 0$.
Denote by $\wtxps$ the diffeomorphism
\begin{equation}
	\psi \times \id_{\cps^{n - 3}} \colon (R_{8} \times \cps^{1}) \times \cps^{n - 3}
		\to (B \times \cps^{1}) \times \cps^{n - 3}.
\end{equation}
Then $\wtxps^{\ast} ([\wtxo]) = [\omega]$
since $\psi^{\ast}([\wtxo^{\prime}]) = [\omega_{1}] + [\omega_{2}]$.
Hence there exists a connection form $\wtxth$ of $P_{M}$
and the lifted CR structure $\wtJ$
such that $d \wtxth / 2 \pi = \wtxps^{\ast} \wtxo$.
We derive from \cref{prop:circle-bundle-over-Hodge-manifold} that
$(P_{M}, \wtH \coloneqq \Ker \wtxth, \wtJ)$ has negative CR Yamabe constant.

Before the proof of \cref{thm:inequivalent-CR-manifold},
we recall some facts on contact geometry.
Two cooriented contact manifolds $(X, H)$ and $(X^{\prime}, H^{\prime})$ are isomorphic
if there exists a diffeomorphism $\psi \colon X \to X^{\prime}$
preserving contact structures and coorientation.
Moreover,
the first Chern class of a strongly pseudoconvex CR manifold
is an invariant of the underlying cooriented contact structure;
see~\cite{Geiges2008}*{Section 2.4} for example.

\begin{proof}[Proof of \cref{thm:inequivalent-CR-manifold}]
	It remains to show that $(P_{M}, H)$ is not isomorphic to $(P_{M}, \widetilde{H})$
	as cooriented contact manifolds.
	Denote by $p \colon P_{M} \to M$ the projection from $P_{M}$ to $M$.
	Then
	\begin{equation}
		c_{1}(P_{M}, H, J)
		= p^{\ast} c_{1}(M),
		\qquad
		c_{1}(P_{M}, \wtH, \wtJ)
		= p^{\ast} \wtxps^{\ast} c_{1}(\wtM).
	\end{equation}
	Now consider the Gysin exact sequence
	\begin{equation}
	\label{eq:Gysin-exact-sequence}
		H^{0}(M; \bbZ) = \bbZ \xrightarrow{- [\omega]} H^{2}(M; \bbZ)
		\xrightarrow{p^{\ast}} H^{2}(P_{M}; \bbZ).
	\end{equation}
	We first consider the case $n = 3$.
	Since $c_{1}(R_{8})$ and $[\omega] = c_{1}(R_{8}) + c_{1}(\cps^{1})$
	are linearly independent in $H^{2}(M; \bbZ)$,
	we have
	\begin{gather}
		c_{1}(P_{M}, H, J)
		= p^{\ast} c_{1}(M)
		= 0, \\
		c_{1}(P_{M}, \wtH, \wtJ)
		= p^{\ast} \wtxps^{\ast} c_{1}(\wtM)
		= - 2 p^{\ast} c_{1}(R_{8})
		\neq 0.
	\end{gather}
	Hence $(P_{M}, H)$ is not isomorphic to $(P_{M}, \wtH)$.

	In the remainder of the proof,
	we assume that $n \geq 4$.
	In this case,
	\begin{equation}
		c_{1}(P_{M}, H, J)
		= p^{\ast} c_{1}(M)
		= - (n - 1) (n - 2) p^{\ast} c_{1}(\calO_{\cps^{n - 3}}(1)).
	\end{equation}
	In particular,
	$[c_{1}(P_{M}, H, J)] = 0$ in $H^{2}(P_{M}; \bbZ) / (n - 1) H^{2}(P_{M}; \bbZ)$.
	It suffices to show that $[c_{1}(P_{M}, \wtH, \wtJ)] \neq 0$
	in $H^{2}(P_{M}; \bbZ) / (n - 1) H^{2}(P_{M}; \bbZ)$.
	Suppose to the contrary that
	$[c_{1}(P_{M}, \wtH, \wtJ)] = [p^{\ast} \wtxps^{\ast} c_{1}(\wtM)] = 0$
	in $H^{2}(P_{M}; \bbZ) / (n - 1) H^{2}(P_{M}; \bbZ)$.
	Consider the following exact sequence:
	\begin{equation}
		H^{0}(M; \bbZ) = \bbZ \xrightarrow{- [\omega]} H^{2}(M; \bbZ) / (n - 1) H^{2}(M; \bbZ)
		\xrightarrow{p^{\ast}} H^{2}(P_{M}; \bbZ) / (n - 1) H^{2}(P_{M}; \bbZ).
	\end{equation}
	This yields that there exists $k \in \bbZ$ such that
	\begin{equation}
		\wtxps^{\ast} c_{1}(\wtM) + k [\omega]
		= - c_{1}(R_{8}) + c_{1}(\cps^{1}) + c_{1}(\cps^{n - 3}) + k [\omega]
		\in (n - 1) H^{2}(M; \bbZ)
	\end{equation}
	Hence
	\begin{equation}
		\langle - c_{1}(R_{8}) + c_{1}(\cps^{1}) + c_{1}(\cps^{n - 3}) + k [\omega], a \rangle \equiv 0
		\quad \textrm{mod} \ n - 1
	\end{equation}
	for any $a \in H_{2}(M; \bbZ)$.
	Let $E \cong \cps^{1}$ be an exceptional divisor in $R_{8}$
	and consider it as a complex curve in $M$.
	Taking $a = [E]$ gives that
	\begin{equation}
		0
		\equiv \langle - c_{1}(R_{8}) + c_{1}(\cps^{1}) + c_{1}(\cps^{n - 3}) + k [\omega], [E] \rangle
		= k - 1
		\quad \textrm{mod} \ n - 1.
	\end{equation}
	Consider also a projective line $L \subset \cps^{n - 3}$,
	which is seen as a complex curve in $M$.
	Then
	\begin{align}
		0
		\equiv \langle - c_{1}(R_{8}) + c_{1}(\cps^{1}) + c_{1}(\cps^{n - 3}) + k [\omega], [L] \rangle
		&= (n - 2) + k n (n - 2) \\
		&\equiv n - 3 \quad \textrm{mod} \ n - 1,
	\end{align}
	which is a contradiction.
\end{proof}

\section*{Acknowledgements}
The first author would like to thank Professor Pak Tung Ho for a kind discussion at KIAS,
and Professor Chong-Kyu Han for introducing him to CR geometry and warm encouragement on this work.

\bibliography{my-reference,my-reference-preprint}

\end{document}